\documentclass[12pt]{amsart}
\usepackage{amssymb,amsmath,amsthm}
\numberwithin{equation}{section}
\begin{document}

\theoremstyle{plain}
\newtheorem{theorem}{Theorem}[section]
\newtheorem{lemma}[theorem]{Lemma}
\newtheorem{proposition}[theorem]{Proposition}
\newtheorem{corollary}[theorem]{Corollary}
\newtheorem{conjecture}[theorem]{Conjecture}

\def\mod#1{{\ifmmode\text{\rm\ (mod~$#1$)}
\else\discretionary{}{}{\hbox{ }}\rm(mod~$#1$)\fi}}

\theoremstyle{definition}
\newtheorem*{definition}{Definition}

\theoremstyle{remark}
\newtheorem{remark}{Remark}[section]
\newtheorem{example}{Example}[section]
\newtheorem*{remarks}{Remarks}
\newcommand{\ndiv}{\hspace{-4pt}\not|\hspace{2pt}}
\newcommand{\cc}{{\mathbb C}}
\newcommand{\qq}{{\mathbb Q}}
\newcommand{\rr}{{\mathbb R}}
\newcommand{\nn}{{\mathbb N}}
\newcommand{\zz}{{\mathbb Z}}
\newcommand{\pp}{{\mathbb P}}
\newcommand{\al}{\alpha}
\newcommand{\be}{\beta}
\newcommand{\ga}{\gamma}
\newcommand{\ze}{\zeta}
\newcommand{\om}{\omega}
\newcommand{\mz}{{\mathcal Z}}
\newcommand{\mi}{{\mathcal I}}
\newcommand{\ep}{\epsilon}
\newcommand{\la}{\lambda}
\newcommand{\de}{\delta}
\newcommand{\De}{\Delta}
\newcommand{\Ga}{\Gamma}
\newcommand{\si}{\sigma}
\newcommand{\Exp}{{\rm Exp}}
\newcommand{\legen}[2]{\genfrac{(}{)}{}{}{#1}{#2}}
\def\End{{\rm End}}

\title{Binary forms with three different relative ranks}

\author{Bruce Reznick}
\address{Department of Mathematics, University of 
Illinois at Urbana-Champaign, Urbana, IL 61801} 
\email{reznick@illinois.edu}

\author{Neriman Tokcan}
\address{Department of Mathematics, University of 
Illinois at Urbana-Champaign, Urbana, IL 61801} 
\email{tokcan2@illinois.edu}
\date{\today}
\begin{abstract}
Suppose $f(x,y)$ is a binary form of degree $d$ with coefficients in
a  field $K \subseteq \cc$. The  {\it $K$-rank of $f$} is the
  smallest number of $d$-th powers of linear forms over $K$ of which $f$ is a
  $K$-linear combination. We prove that for $d \ge 5$, there always
  exists a form of degree $d$ with at least three different ranks
  over various fields. The $K$-rank of a form $f$ (such as $x^3y^2$) may
  depend on whether -1 is a sum of two squares in $K$.
\end{abstract}

\maketitle
 \section{Introduction}
Suppose $f(x,y)$ is a binary form of degree $d$ with coefficients in
a  field $K \subseteq \cc$. The {\it
$K$-length} or {\it $K$-rank of $f$}, $L_K(f)$, is the smallest $r$
for which there is a 
representation 
\begin{equation}\label{E:basic}
f(x,y) = \sum_{j=1}^r \la_j\bigl(\al_{j}x + \be_jy\bigr)^d
\end{equation}
with $ \la_j, \al_j, \be_j \in K$.
In case $K = \mathbb C$ or $\mathbb R$, these are
commonly called the {\it Waring rank} or {\it real Waring rank}. We
shall say that two linear forms are {\it  distinct} if they (or their
$d$-th powers)  
are not proportional. A  representation such as \eqref{E:basic}  is
{\it   honest} if the summands are pairwise distinct; that is, if
$\la_i\la_j(\al_i\be_j-\al_j\be_i) \neq 0$ whenever $i \neq j$. Any
representation in which $r = L_K(f)$ is necessarily honest. 

Of course, if $K \subseteq F \subseteq \cc$, then $f \in F[x,y]$ as well, and
one may consider $L_F(f)$ to be the {\it relative rank} of $f$ with
respect to $F$. It is not hard to find forms with two different
relative ranks; for example, suppose $\ga\notin \qq, \ga^2 \in \qq$ and 
$f(x,y) = (x +\ga y)^d + (x- \ga y)^d \in \qq[x,y]$. Then $L_F(f)$ is
equal to 2 if $\ga \in F$ and $d$ otherwise; see \cite[Thm.4.6]{Re1}.

A form $h \in \cc[x,y]$ is {\it apolar} to $f$ if
$h(D)f = h(\frac{\partial}{\partial x}, \frac{\partial}{\partial
  y})f(x,y) = 0$. If \eqref{E:basic} holds and is honest, then $\prod (\be_j x -
\al_j y)$ is apolar to $f$. Sylvester showed (\cite{S1,S2}, see
Theorem 2.1 below) that this is an if
and only if condition: if $\prod (\be_j x - \al_j y)$ is square-free,
and apolar to $f$, then there exist $\la_j \in \cc$ making  \eqref{E:basic}
true. In \cite{Re1}, the first author showed that the same statement
is true if $\al_j,\be_j,\la_j$ are restricted to be in any field $F
\subseteq \cc$. The computation of determining whether
$h$ is apolar to $f$ is equivalent to the Sylvester algorithm. The set
of all forms which are apolar to a given form 
$f$ is called its {\it apolar ideal} and it is known, see
\cite[Thm.1.44(iv)]{IK}, \cite[Lemma 2.8]{BS} that this is generated
by two relatively prime forms, the sums of whose degrees is $d+2$.
An upper bound on the relative rank is given by \cite[Thm.4.10]{Re1}: if $f \in
K[x,y]$, then $L_K(f) \le \deg f$. 

Sylvester also proved (\cite{S3}, see Corollary 2.3 below), a variation on
Descartes' Rule of Signs for univariate polynomials, which can be
extended (\cite{Re1}) to binary forms. Suppose $f \in \rr[x,y]$ is
{\it hyperbolic}; that is, $f$
splits over $\rr$ ($f$ needn't be square-free), but $f$ is not a
$d$-th power. Then $L_{\rr}(f) = \deg
f$. It was conjectured (and proved for $d \le 4$) in \cite{Re1} that the
converse is also true:  if
$L_{\rr}(f) = \deg f$, then $f$ is hyperbolic (and not a $d$-th power.) This was
proved by  Causa and Re \cite{CR} and Comon and Ottaviani \cite{CO}
when $f$ is square-free, and very recently, unconditionally, by
Blekherman and Sinn \cite[Thm 2.2]{BS}.

The first author showed in \cite{Re1} that for $\phi(x,y) = 3x^5
-20x^3y^2+10xy^4$, we have $L_K(\phi) = 3$ if and only if $\sqrt{-1} \in K$,  
$L_K(\phi)= 4$ for
$K=\qq(\sqrt{-2}),\qq(\sqrt{-3}),\qq(\sqrt{-5}),$ $\qq(\sqrt{-6})$ (at least)
and $L_{\rr}(\phi) = 5$. This example also shows (by taking $K_1 =
\qq(\sqrt{-2})$ and $K_2 = \qq(\sqrt{-3})$)  that $L_{K_1}(f) =
L_{K_2}(f) <  L_{K_1\cap K_2}(f)$ is possible.  Furthermore
(\cite[Cor.5.1]{Re1}), if $f$ has $k$ different ranks, then $\deg f
\ge 2k-1$; so three different ranks 
cannot occur for forms of degree $\le 4$.

The main result of this paper is that in all degrees $d \ge 5$,
 there exist
binary forms of degree $d$ with at least three different ranks over
different fields (see Theorem 3.1). In particular, let $\zeta_m$ denote a
primitive $m$-th root of unity. We shall prove
that if $k \ge 3$ and $p_{2k-1}(x,y)= x^{k-1}y^{k-1}(x-y)$, then 
\[
L_{\qq(\zeta_{k+1})}(p_{2k-1}) = k,\quad L_{\qq(\zeta_{k})}(p_{2k-1}) = k+1,\quad
 L_{\rr}(p_{2k-1}) = 2k-1 > k+1.
\]  
Similarly,
 if $k \ge 3$ and $p_{2k}(x,y)= x^ky^k$, then 
\[
L_{\qq(\zeta_{k+1})}(p_{2k}) = k+1,\quad L_{\qq(\zeta_{k})}(p_{2k}) = k+2,\quad
   L_{\rr}(p_{2k}) = 2k > k+2.
\]
We are not aware of any binary form of any degree with more
than three different ranks. We do not consider forms in more
than two variables in this paper. 

The relative rank can depend on algebraic properties of
the underlying field. The {\it Stufe} of a non-real field $F$, $s(F)$,
is the smallest 
integer $n$ such that $-1$ can be written as a sum of $n$ squares in
$F$. It is already known that $L_{\cc}(x^3y^2) = 4$ 
(from \cite[Prop.3.1]{CCG}) and $L_{\rr}(x^3y^2) = 5$ 
(from \cite[Prop.4.4]{BCG}). We show in Theorem 4.1 that $L_K(x^3y^2)
= 4$ if and only if $s(K) \le 2$ and $L_K(x^3y^2) = 5$ otherwise. (For
more on the real rank of monomials, see \cite{CKOV}.) We show in Theorem 4.2
that if $m$ is a square-free positive integer and  $f(x,y) = \binom 61
x^5y - \binom 63 x^3y^3$, then $L_{\qq(\sqrt{-m})}(f) = 4$ if and only
if $s(\qq(\sqrt{-m})) = 2$ if and only if $m \not\equiv 7 \mod 8$ (see
\cite{N,Sz}), and
$L_{\qq(\sqrt{-7})}(f) = 5$.

We happily acknowledge useful conversations with T. Y. Lam and Steve
Ullom and are extremely grateful to Gerry Myerson for pointing out the
references \cite{N, Sz}.

Part of the work in this paper is taken from the doctoral dissertation
of the second author, being written under the direction of the first
author.  The first author was supported in part by
Simons Collaboration Grant 280987.

\section{Tools}
The following theorems are proved in \cite{Re1} and, for $K = \cc$,
are due to Sylvester \cite{S1,S2} in 1851.

\begin{theorem}\cite[Thm.2.1,Cor.2.2]{Re1}
Suppose $K \subseteq \cc$ is a field, 
\begin{equation}\label{E:S1}
f(x,y) = \sum_{j=0}^d \binom dj a_j x^{d-j}y^j \in K[x,y]
\end{equation}
 and suppose $r \le d$ and
\begin{equation}\label{E:S2}
h(x,y)
= \sum_{t=0}^r 
c_tx^{r-t}y^t = \prod_{j=1}^{r} ( - \beta_j x + \alpha_j y)
\end{equation}
 is  a product of pairwise distinct linear factors, with
 $\alpha_j,\beta_j \in K$.    Then there exist
 $\lambda_j\in K$ so that  
\begin{equation}\label{E:S3}
f(x,y) = \sum_{j=1}^r \lambda_j (\alpha_j x + \beta_j y)^d
\end{equation}
if and only if
\begin{equation}\label{E:S4}
\begin{pmatrix}
a_0 & a_1 & \cdots & a_r \\
a_1 & a_2 & \cdots & a_{r+1}\\
\vdots & \vdots & \ddots & \vdots \\
a_{d-r}& a_{d-r+1} & \cdots & a_d
\end{pmatrix}
\cdot
\begin{pmatrix}
c_0\\c_1\\ \vdots \\ c_r
\end{pmatrix}
=\begin{pmatrix}
0\\0\\ \vdots \\ 0
\end{pmatrix};
\end{equation}
that is, if and only if 
\begin{equation}\label{E:S5}
\sum\limits_{t=0}^r a_{\ell + t} c_t = 0, \qquad \ell = 0,1,\dots,  d-r.
\end{equation}
\end{theorem}

We remark that if $f$ and $h$ are defined by \eqref{E:S1} and
\eqref{E:S2}, then
\[
h(D)f = \sum_{m=0}^{d-r} \frac{d!}{(d-r-m)!m!}\left( \sum_{i=0}^r
  a_{i+m}c_i \right) x^{d-r-m}y^m.
\]
Thus, Theorem 2.1 provides an algorithm for determining the forms
of a given degree $\le \deg f$ which are apolar to $f$. If
\eqref{E:S4} holds and 
$h$ is square-free, 
then we say that $h$ is a {\it   Sylvester form of degree $r$} for
$f$ over $K$. In other words, $L_K(f) = r$ if and only if 
there is a Sylvester form for $f$ over $K$ of degree $r$.

The next  theorem is a generalization of a 1864 theorem of
Sylvester \cite{S3}; the original applied to real polynomials in one
variable and was adapted to real binary forms in \cite{Re1}.

\begin{theorem}\cite[Thm.3.1,3.2]{Re1}
Suppose $f(x,y)$ is a non-zero real form of degree $d$ with $\tau$
real linear factors (counting multiplicity), $f$ is not the $d$-th
power of a linear form  and 
\begin{equation}  \label{E:sylrep}
f(x,y) = \sum_{j=1}^r \la_j(\cos \theta_j x + \sin \theta_j y)^d,
\end{equation}
where $-\frac {\pi}2 < \theta_1 < \dots < \theta_r \le \frac{\pi}2$,
$r \ge 2$  and $\la_j \neq 0$. If there are $\sigma$ sign changes
in the tuple $(\la_1,\la_2,\dots,\la_r,(-1)^d \la_1)$, then $\tau \le
\sigma$. In particular, $\tau \le r$.
\end{theorem}

\begin{corollary}\cite[Cor.4.11]{Re1}
If $f \in \rr[x,y]$ is a product of $d$ real linear forms and not a
$d$-th power, then $L_{\rr}(f) = d$. 
\end{corollary}

\begin{corollary}
If $f \in \rr[x,y]$ is a product of $d$ real linear forms and not a
$d$-th power, and $g \in \rr[x,y]$ is apolar to $f$, with $\deg g <
d$, then $g$ cannot be square-free.
\end{corollary}
We shall also need the following result from \cite{Re1}.
\begin{theorem}\cite[Thm.4.10]{Re1}
If $f \in K[x,y]$, then $L_K(f) \le \deg f$.
\end{theorem}

The next tool is an exercise
in a first course in algebraic number theory. We include the proof for
completeness. (See \cite[p.158(Lemma 3)]{BSh} for a more incisive, but less
elementary, proof.)
Recall that $\ze_d = e^{\frac{2\pi i}d}$.
\begin{theorem}
Suppose $m,n$ are integers. Then $\ze_m \in \qq(\ze_n)$ if and only if $m \ | \
n$ or $n$ is odd and $m \ | \ 2n$. 
\end{theorem}
\begin{proof}
Note that $\zeta_m = \zeta_{mt}^t$. If $n$ is odd and $m$ divides $2n$
but not $n$, then  $m = 2u$ and $n= tu$ with odd $t,u$, so
$\zeta_m = \zeta_{2n}^t = -\zeta_{2n}^{t+tu} = -\zeta_n^{t(u+1)/2}
\in\qq(\ze_n)$.

Conversely, let $g = gcd(m,n)$ so that $m = gr, n = gs$, where
$gcd(r,s) = 1$, and let $q = grs = lcm(m,n)$.  Then $\zeta_m =
\zeta_q^s$ and $\zeta_n = \zeta_q^r$. Now
choose integers $e,f$ so that $es + fr = 1$. We  have
$\ze_m^e\ze_n^f=\ze_q^{es+fr} = \ze_q$. Since $\ze_m \in \qq(\ze_n)$, it
follows that $\ze_q \in \qq(\ze_n)$, so $\qq(\ze_q) \subseteq
\qq(\ze_n)$, but since $n \ | \ q$, the converse inclusion holds as
well, and so $\qq(\ze_q) = \qq(\ze_n)$. This in turn implies that
$\Phi(n) = \Phi(q)$. Since $n \ | \ q$, this implies that $n = q$ (and
$gs=grs$, so $r=1$ and $m \ | \ n$) or $n$ is odd and $q = 2n$ (and
$grs=2gs$, so $r=2$ and $m \ | \ 2n$).
\end{proof}
\begin{corollary}
If $m \ge 3$, then  $\ze_m \not\in \qq(\ze_{m \pm 1})$.
\end{corollary}

For our final result, we make a minor gloss on the work of Nagell \cite{N};
see also the beautiful short proof of Szymiczek \cite{Sz}.

\begin{theorem}
Suppose $F = \qq(\sqrt{-m})$, where $m$ is a square-free positive
integer. Then there exist solutions to either of the equations
\begin{equation}\label{E:2.7}
r^2 + s^2 = -1,\qquad rs(r^2-s^2) \neq 0, \quad r, s \in F
\end{equation}
\begin{equation}\label{E:2.8}
t^2 + u^2 = -2,\qquad tu(t^2-u^2) \neq 0, \quad t,u  \in F
\end{equation}
if and only if $m \not\equiv 7 \mod 8$.
\end{theorem}
\begin{proof}
First note that if \eqref{E:2.7} holds and $(t,u) = (r+s,r-s)$, then
$t^2+u^2 = 2(r^2+s^2)=-2$ and $tu(t^2-u^2) = 4rs(r^2-s^2)$, so
\eqref{E:2.8} holds. This argument goes the other way with $(r,s) =
(\frac{t+u}2,\frac{t-u}2)$, and so it suffices to
prove the theorem for  \eqref{E:2.7}.

Nagell \cite{N} proves that $s(\qq(\sqrt{-m}))=2$ (that is, there is a
solution to $r^2+s^2=-1$ in 
$\qq(\sqrt{-m})$) if and only if  $m \not\equiv 7 \mod 8$, so all we
need to do is consider the additional condition $ rs(r^2-s^2) \neq
0$. If $r^2+s^2=-1$ and $rs(r^2-s^2) = 0$, then  up to permutation,
$(r,s) = (\pm i,0)$ or 
$(\pm \frac{\sqrt{-2}}2,\pm \frac{\sqrt{-2}}2)$. These solutions are
relevant to $\qq(\sqrt{-m})$ only  when $m=1,2$, in which case the
following alternatives suffice:   
\[
\qq(\sqrt{-1}): \quad \left(\tfrac 34\right)^2 + \left(\tfrac
{5i}4\right)^2 = -1, \qquad 
\qq(\sqrt{-2}): \quad 7^2 + (5\sqrt{-2})^2 = -1.
\]
\end{proof}

\section{Three ranks}
Our general strategy is straightforward. Suppose $d=2k-1$ is
odd. Choosing $r=k$, we see that \eqref{E:S4} is a $k \times (k+1)$
linear system, which in general has a unique solution. We consider
a form $f$ of degree $2k-1$ which is a product of real linear factors,
so $L_{\rr}(f) = 2k-1$. We also choose $K$ to be the field generated
by the coefficients of this unique representation of $f$ over $\cc$,
so $L_K(f) = k$; necessarily, $f \in K[x,y]$.
 Finally, we somehow find a representation 
of rank between $k$ and $2k-1$ over a non-real field which does not
contain the rank $k$ representation.  If $d=2k$, the same heuristic
applies, but there will 
be, in general, infinitely many representations of rank $k+1$. In
certain cases though, each of these representations must contain a specific
non-real root of unity $\ze$.
\begin{theorem}
If $d \ge 5$, then there exists a binary form $p_d$ of degree $d$ which
takes at least three different ranks.
\end{theorem}
\begin{proof}
Let $p_{2k-1} = \binom{2k-1}k x^{k-1}y^{k-1}(x-y)$, so that in
\eqref{E:S4}, $a_{k-1} = 1$, $a_k=-1$ and $a_i = 0$ otherwise. 
First, with $r=k-1$, we see that the matrix from \eqref{E:S4} is non-singular:
\[
\begin{pmatrix}
0 & 0 &\cdots  &  0 & 1  \\
0 & 0 & \cdots  & 1 & -1  \\
\vdots & \vdots  & \vdots &\vdots & \vdots \\
1& -1 & \cdots  & 0 & 0 \\
-1 & 0 & \cdots   & 0 & 0
\end{pmatrix}.
\]
It follows that
$L_{\cc}(p_{2k-1}) > k-1$. On taking
$r=k$,  \eqref{E:S4} becomes:
\begin{equation}\label{E:odd}
\begin{pmatrix}
0 & 0 &\cdots & 0 & 1 & -1 \\
0 & 0 & \cdots& 1 & -1 & 0 \\
\vdots & \vdots & \ddots & \vdots &\vdots & \vdots \\
1& -1 & \cdots & 0 & 0 & 0
\end{pmatrix}
\cdot
\begin{pmatrix}
c_0\\c_1\\ \vdots \\ c_k
\end{pmatrix}
=\begin{pmatrix}
0\\0\\ \vdots \\ 0
\end{pmatrix}.
\end{equation}
Clearly, the only solution to \eqref{E:odd} has $c_i = c$ for all $i$,
so that up to multiple,
\[
h(x,y) = \sum_{t=0}^k x^{k-t}y^t = \frac{x^{k+1}-y^{k+1}}{x-y} =
 \prod_{j=1}^k (x - \ze_{k+1}^jy),
\]
and so $L_K(p_{2k-1}) = k$ if and only if $\ze_{k+1} \in K$; in
particular,  $L_{\qq(\ze_{k+1})}(p_{2k-1}) = k$. Since $p_{2k-1}$ is
hyperbolic, it follows from Corollary 2.3 that
$L_{\rr}(p_{2k-1}) = 2k-1$. 

Now set $r = k+1$, so that  \eqref{E:S4} becomes:
\begin{equation}\label{E:odd2}
\begin{pmatrix}
0&0 & 0 &\cdots & 0 & 1 & -1 &0\\
0&0 & 0 & \cdots& 1 & -1 & 0 &0\\
\vdots & \vdots & \vdots & \ddots & \vdots & \vdots & \vdots &\vdots \\
0 &1& -1 & \cdots & 0 & 0 & 0 &0
\end{pmatrix}
\cdot
\begin{pmatrix}
c_0\\c_1\\ \vdots \\ c_{k+1}
\end{pmatrix}
=\begin{pmatrix}
0\\0\\ \vdots \\ 0
\end{pmatrix}.
\end{equation}
The system \eqref{E:odd2} implies $c_1= \dots = c_k$,
but places no conditions on $c_0$
and $c_{k+1}$ . In particular, we may choose $c_0=c_{k+1}
= 0$ and  $c_1= \dots = c_k=1$, to get a Sylvester polynomial over
 $\qq(\ze_{k})$:
\[
h(x,y) = \sum_{t=1}^k x^{k+1-t}y^t = xy\left( \frac{x^{k}-y^{k}}{x-y}\right)=
xy\prod_{j=1}^{k-1} (x - \ze_{k}^jy).
\]
It follows that $L_{\qq(\ze_{k})}(p_{2k-1}) \le k+1$. Since $\ze_{k+1}
\notin \qq(\ze_{k})$ by Corollary 2.7, it follows that
$L_{\qq(\ze_{k})}(p_{2k-1}) = k+1$. 

Since
\[
\begin{gathered}
c_0x^{k+1} + c_1 (x^ky+ \dots + xy^k) + c_{k+1}y^{k+1} = \\
(c_0x + (c_1-c_0)y)(x^k + \dots + y^k) + (c_{k+1}-c_1+c_0)y^{k+1},
\end{gathered}
\]
it is not hard to show that the apolar ideal of $p_{2k-1}$ is generated by 
$x^k + \dots + y^k$ and $y^{k+1}$; note that $k + (k+1) = (2k-1) + 2$.
It seems to be a quite difficult question to determine
which fields $K$ have the property that, for a suitable choice of
$c_i$'s,  this form is  square-free and
splits over $K$. We return to this
type of question in the next section.

Now suppose that $p_{2k}(x,y) = \binom{2k}k x^ky^k$, so $a_k=1$ and
$a_i=0$ otherwise. (This example is also discussed in
\cite[Thm.5.5]{Re1}.) Taking $r=k$, we note that the matrix 
\begin{equation}
\begin{pmatrix}
0  & \cdots & 0 & 1\\
0  & \cdots & 1 & 0\\
 \vdots & \ddots & \vdots & \vdots \\
 1 & \cdots & 0 & 0 
\end{pmatrix}
\end{equation}
is nonsingular, hence there are no representations of rank $k$. For
$r=k+1$,

\begin{equation}
\begin{pmatrix}
0 & 0 & \cdots & 0 & 1 & 0 \\
0 & 0 & \cdots & 1 & 0 & 0 \\
\vdots & \vdots & \ddots & \vdots & \vdots &\vdots \\
0 & 1 & \cdots & 0 & 0 & 0
\end{pmatrix}
\cdot
\begin{pmatrix}
c_0\\c_1\\ \vdots \\ c_{k+1}
\end{pmatrix}
=\begin{pmatrix}
0\\0\\ \vdots \\ 0
\end{pmatrix}
\implies c_1 = \cdots =c_k = 0. 
\end{equation}
Thus every Sylvester form of degree $k+1$ has the shape $h(x,y) = \al x^{k+1} -
\be y^{k+1}$ and the apolar ideal of $p_{2k}$ is generated by
$x^{k+1}$ and $y^{k+1}$. 
If $h$ has distinct factors, then $\al\be \neq 0$
and
\[
h(x,y) = \al \prod_{j=0}^k (x - \ze_{k+1}^ju y),
\]
where $\al u^{k+1} = \be$. If $h$ splits over $K$, then $u, \ze_{k+1}u
\in K$, hence $\ze_{k+1} \in K$ and $\qq(\ze_{k+1}) \subseteq K$. In
particular, by taking $\al = \be = 1$, we see that $x^{k+1} - y^{k+1}$
is a Sylvester form for $p_{2k}$ over  $\qq(\ze_{k+1})$, and so
$L_{\qq(\ze_{k+1})}(p_{2k}) = k+1$. Since $x^ky^k$ is hyperbolic, 
 $L_{\rr}(p_{2k}) = 2k$. 

Any expression of rank $k+2$ over $K$
would have a Sylvester form of shape
\[
(\al x+ \be y)x^{k+1} + (\ga x+ \de y)y^{k+1}.
\]
In particular,  $xy(x^k-y^k)$ splits over
 $\qq(\ze_{k})$, which does not contain $\ze_{k+1}$ and so we have
 $L_{\qq(\ze_{k})}(p_{2k})= k+2$. 
\end{proof}

Here are explicit representations of $p_5$, $p_6$
and $p_7$ as sums of powers of linear forms.

\begin{example} For $k=3$, the following two formulas may be directly
  verified (as usual, $\om = \ze_3$ and $i=\ze_4$):
\[
\begin{gathered}
p_5(x,y) = 10x^2y^2(x-y)  \\
=\tfrac 14 \cdot \left( (-1-i)(x+iy)^5 +
  2(x-y)^5 + (-1+i)(x-iy)^5\right) \in \qq(\ze_4)[x,y]
\\
= x^5 - y^5 + \tfrac 1{\om-\om^2} \cdot \left( \om^2(x + \om y)^5
  -\om(x + \om^2 y)^5 \right)\in \qq(\ze_3)[x,y].
\end{gathered}
\]
The expressions seem to get more complicated for larger values of $k$.
For example,
\[
\begin{gathered}
(1+2\ze_5 +3\ze_5^2-\ze_5^3)p_7(x,y) = \\\ze_5^4 (x + \ze_5 y)^7 -
\ze_5^2(1+\ze_5+\ze_5^2)(x + \ze_5^2 y)^7 + \ze_5(1+\ze_5+\ze_5^2)(x
+ \ze_5^3 y)^7 - \ze_5(x + \ze_5^4 y)^7.
\end{gathered}
\]
Here, $1+2\ze_5 +3\ze_5^2-\ze_5^3 = i\sqrt{\frac 52 (5+\sqrt5)} \approx 4.25i$.
\end{example}
\begin{example} The representations of $p_{2k}$ of rank $k+1$ are
  given in \cite[Thm.5.5]{Re1}. For $k=3$, taking $w=1$ in
  \cite[(5.6)]{Re1}, we obtain   after some simplification,
\[
\begin{gathered}
p_6(x,y) = 20x^3y^3 
\\
= \tfrac 14 \cdot \left((x+y)^6 +
  i(x+iy)^6-(x-y)^6-i(x-iy)^6 \right) \in \qq(\ze_4)[x,y]
\\
=  \tfrac 13 \cdot \left((x+y)^6 + (x+\om y)^6 +
  (x+\om^2y)^6-3x^6 -3y^6 \right)\in  \qq(\ze_3)[x,y].
\end{gathered}
\]
The evident patterns shown above are easily proved, using the methods
of \cite{Re1}. 
\end{example}

\section{Two more examples}

In this section, we give some additional examples, in which
\eqref{E:S4} is altered only slightly from $p_5$ and $p_6$, but the
results show a sensitivity to arithmetic conditions. Recall that $s(K)
\le 2$ means that there exist $r,s \in K$ so that $r^2+s^2=-1$.

\begin{theorem}
Suppose $f(x,y) = \binom 52 x^3y^2$. Then $L_K(f) = 4$ iff $s(K) \le
2$; otherwise, $L_K(f) = 5$.
\end{theorem}
\begin{proof}
We already know from \cite{CCG} that $L_{\cc}(f) = 4$, hence $L_K(f)
\ge 4$. This can also be shown directly via Theorem 2.1. We omit the
details. 

Suppose now that $L_K(f) = 4$.  Then, 
\begin{equation}
\begin{pmatrix}
0 & 0 & 1 & 0 & 0 \\
0 & 1 & 0 & 0 & 0\\
\end{pmatrix}
\cdot
\begin{pmatrix}
c_0\\c_1\\c_2 \\ c_3\\c_4
\end{pmatrix}
=\begin{pmatrix}
0\\0
\end{pmatrix}
\implies c_1 = c_2=0
\end{equation}
and $h(x,y) = c_0x^4 + c_3xy^3 + c_4y^4$ is a Sylvester form for $f$
over $K$.  Thus, we are led to the
question: for which choices of $c_i$ and which fields $K$ can such a
square-free form 
split into distinct factors over $K$? 

 If $c_0=0$ then $h$ is not square-free, so we scale to
take $c_0=1$. Then $L_K(f) = 4$ if and only if there exist distinct $r_i \in K$
so that
\[
x^4 + c_3xy^3+c_4y^4 = (x-r_1y)(x-r_2y)(x-r_3y)(x-r_4y);
\]
that is, if and only if the Diophantine system
\begin{equation}\label{E:dioph}
r_1+r_2+r_3+r_4 = r_1r_2+r_1r_3+r_1r_4+r_2r_3+r_2r_4+r_3r_4 = 0
\end{equation}
has a solution in $K$ with distinct $r_i$'s.

We solve \eqref{E:dioph}, first ignoring the restriction to distinct
elements. Putting $r_4 = -(r_1+r_2+r_3)$ into the second equation yields
\[
\begin{gathered}
r_1^2 + r_2^2 + r_3^2 + r_1r_2 + r_1r_3+r_2r_3 = 0 \implies \\
r_3 = -\frac{r_1+r_2}2 \pm \frac {\sqrt{-3 r_1^2 - 2 r_1r_2 - 3r_2^2}}2.
\end{gathered}.
\]
Choose $r_1,r_2 \in K$. We see that $r_3 \in K$ (and so $r_4 \in K$)
if and only if 
\[ -3 r_1^2 - 2 r_1r_2 - 3r_2^2 = -2(r_1+r_2)^2 - (r_1-r_2)^2 = w^2
\]
 is a non-zero square in $K$. Let $(X,Y,Z) =
(w,r_1-r_2,r_1+r_2) \in K^3$. We have (as in the proof of Theorem 2.8)
\[
-2Z^2-Y^2=X^2 \implies \left(\frac XZ \right)^2 + \left(\frac YZ
\right)^2 = -2 \implies  \left(\frac {X+Y}{2Z} \right)^2 + \left(\frac {X-Y}{2Z}
\right)^2 = -1.
\]
Thus, if  $L_K(f) = 4$, then $s(K) \le 2$. The converse is almost immediate.

If \eqref{E:dioph} has repeated $r_i$'s, we may assume without loss of
generality that $r_1=r_2$, hence $r_3,r_4 = r_1(-1 \pm \sqrt{-2})$. 
The only fields in which this solution might occur contain $\sqrt{-2}$,
so if we can find an alternate solution to \eqref{E:dioph} in
$\qq(\sqrt{-2})$, we will be done. It may be checked that
\[
\{r_1,r_2\} = \{5\sqrt{-2}\pm 6\}, \qquad \{r_3,r_4\} = \{-5\sqrt{-2}\pm 8\}
\]
is such an alternate solution to \eqref{E:dioph} with distinct $r_i$.
\end{proof}

Our final result presents another sextic with three different ranks.
\begin{theorem}
Suppose $f(x,y) = \binom 61 x^5y - \binom 63 x^3y^3
=2x^3y(3x^2-10y^2)$. Then $L_K(f) = 
4$ if and only if $s(K) \le 2$. In particular, if $m$ is a positive
square-free integer, and $m \not\equiv 7 \mod 8$, then
$L_{\qq(\sqrt{-m})}(f) = 4$. Further, $L_{\qq(\sqrt{-7})}(f) = 5$.
\end{theorem}
\begin{proof} 
Again, taking \eqref{E:S4} for $r=3$ gives a nonsingular matrix
\begin{equation}
\begin{pmatrix}
0 & 1 & 0 & -1 \\
1 & 0 & -1 & 0\\
0 & -1 & 0 & 0\\
-1 & 0 & 0 & 0
\end{pmatrix},
\end{equation}
so $L_{\cc}(f) > 3$. Moving up one,
\begin{equation}
\begin{pmatrix}
0 & 1 & 0 & -1 &0 \\
1 & 0 & -1 & 0 & 0\\
0 & -1 & 0 & 0& 0\\
\end{pmatrix}
\cdot
\begin{pmatrix}
c_0\\c_1\\c_2 \\ c_3 \\ c_4
\end{pmatrix}
=\begin{pmatrix}
0\\0\\0
\end{pmatrix}
\implies c_0=c_2,\quad c_1=c_3 = 0,
\end{equation}
so the possible Sylvester polynomials over $K$ have the shape
$h(x,y) = c_0 x^4 + c_0 x^2y^2 + c_4 y^4.$
If $c_0c_4=0$, then $h$ is not square-free, so we may scale to $c_0 =
1$. Since $h$ is an even polynomial, if $x - ry$ is a factor with $r
\neq 0$ (since $c_4 \neq 0$), then 
so is $x+ry$, hence if $h$ splits over $K$, then there exist $r,s
\in K$ ($r^2 \neq s^2 \neq 0$) so that
\[
x^4 + x^2y^2 + c_4y^4 = (x^2 - r^2y^2)(x^2 - s^2y^2).
\]
Thus, $L_K(f) = 4$ if and only if $K$ is a field in which the equation
\begin{equation}\label{E:4.7}
r^2 + s^2 = -1
\end{equation}
has a solution, $r^2 \neq -\frac 12,0,-1$. As we have seen in the
proof of Theorem 2.8,
this is true precisely  when $s(K) \le 2$, so if $K = \qq(\sqrt{-m})$,
precisely when  $m \not\equiv 7 \mod 8$.

Since  $f$ is hyperbolic, $L_{\rr}(f) = 6$. The previous paragraph
shows that the apolar ideal for $f$ is generated by $x^4+x^2y^2$ and
$y^4$. 
We now wish to find at least one field $K$ for which  $L_K(f) =
5$. Since $K$ must be non-real with $s(K) >2$, we take $K =
\qq(\sqrt{-7})$ and look for a representation with relative rank
5. To this end, observe that $y(x^4+x^2y^2)-2x y^4 = x^4y +x^2y^3 -
2xy^4  
 = xy(x-y)\left(x + \tfrac{1 + \sqrt{-7}}2 y \right) \left(x + \tfrac{1
    - \sqrt{-7}}2 y \right)$
splits over $\qq(\sqrt{-7})$. 
\end{proof}

\end{document}